\documentclass[11pt]{amsart}

\newcommand{\comment}[1]{}

\usepackage{a4wide}
\usepackage{times}
\usepackage{bbm}
\usepackage{mathtools, amssymb}
\usepackage{graphicx,xspace}
\usepackage{epsfig}
\usepackage{dsfont}
\usepackage[usenames,dvipsnames]{xcolor}
\usepackage{tikz}
\usepackage[T1]{fontenc}
\usepackage[utf8]{inputenc}
\usepackage{array,multirow} %
\usepackage{bm}
\definecolor {processblue}{cmyk}{0.96,0,0,0}

\usepackage{kpfonts} 
\usepackage{dsfont} 
\usepackage{cancel}
\usepackage{tipa}
\usepackage{stmaryrd}
\usepackage{csquotes}

\usepackage{hyperref,pifont}
\usepackage{todonotes}
\usepackage{dsfont}

\newcommand{\ignorer}[1]{}

\def\and{\ \wedge\ }
\usepackage[capitalize]{cleveref}

\theoremstyle{plain}

\newtheorem{lemma}{Lemma}
\newtheorem{theorem}[lemma]{Theorem}
\newtheorem{corollary}[lemma]{Corollary}
\newtheorem{proposition}[lemma]{Proposition}
\newtheorem{property}[lemma]{Property}

\newtheorem{definition}[lemma]{Definition}

\theoremstyle{remark}

\renewcommand\epsilon{\varepsilon}

\def\card{\text{card}}

\def\({\left(}
\def\){\right)}

\def\P{\mathbb{P}}

\def\E{\mathbb{E}}

\DeclareMathOperator{\LIS}{LIS}  
\DeclareMathOperator{\LDS}{LDS}  
\DeclareMathOperator{\LAS}{LAS}  

\newcommand{\Old}[1]{}

\def\kS{\mathfrak{S}}

 \newcommand{\CI}[1]{(\hyperref[eq:cycle_controle]{\ensuremath{\mathcal{CI}_{#1}}})}

\newcommand\restr[2]{{%
		\left.\kern-\nulldelimiterspace %
		#1 %
		\right|_{#2} %
	}}

\title[LD for random permutations ] {About Universality of large deviation principles 
 for conjugacy invariant permutations }

 \author[A. Guionnet]{Alice Guionnet}
       \email{alice.guionnet@ens-lyon.fr}
 \author[M.S. Kammoun]{Mohammed Slim Kammoun}
    \email{slim.kammoun@ens-lyon.fr}

\begin{document}
\maketitle

\begin{abstract}
We prove the universality of the large deviations for  conjugacy invariant permutations with few cycles. As an application, 
we establish the universality of large deviation principles at speeds $n$ and $\sqrt{n}$ for the length of monotone subsequences in conjugacy invariant permutations, with a  sharp control over the total number of cycles. This universality class includes the well-known Ewens measures. 
\end{abstract}

\section{Introduction and main results}

{De Moivre was the first to show in 1738 that errors from the law of large numbers were described by the bell curve, also known as the Gaussian or the Normal curve. However, his approach was in essence combinatorial and his result only concerned the sum of independent Bernoulli variables. In a tour de force, Laplace \cite{Laplace} generalized this result  and showed that the Gaussian law characterizes the fluctuations of many more random variables. This was the first instance of universality, and a key result in probability theory. Other distributions such as  Poisson processes or  the Gumbel distribution appear as  limits of different models with  weakly correlated structure.
It was more recently shown that models with strongly correlated structures also have universal fluctuations, which are characterized by limiting distributions such as the semicircular law \cite{zbMATH03113773,zbMATH06111055}, Tracy-Widom distribution{\cite{Baik,Joh}, Airy process \cite{tracywidom,LaYa}, sine process \cite{tao,LaYa}, etc., emerge.   On the contrary to fluctuations and the central limit theorem , large deviations theory \cite{DZ} rarely displays a universal feature, the probability to take an unlikely state being in general much more dependent on the underlying model. However, 
recent advances in random matrix theory have shown a surprising universality feature in large deviation principles   \cite{BADG,Myl, GM, BG, GuHu1, Augeri,GHN}.  It was shown in \cite{GuHu1} that the large deviations  for the largest eigenvalues of GUE (Gaussian Unitary Ensemble)  are universal, and the same as those of  matrices with Gaussian entries derived in \cite{BADG},  within a class of matrices with subGaussian entries whose distribution is called  sharp subGaussian because their Laplace transform is uniformly bounded by the Gaussian one with the same variance. \cite{MR4239836,CDG} show that the probability of deviating  towards a sufficiently small value stay universal for all subGaussian entries, but are different for deviations towards large enough value if the entries are not sharp subGaussian. The goal of this article is to study a similar universality phenomenon for random permutations. 

In the last few years, universality  for random permutations have been investigated thoroughly, including various aspects like global convergence and fluctuations. In particular, for conjugacy invariant permutations (see Definition~\ref{Def: CI}), many results suggest that for a large family of functions on permutations, the asymptotic behavior depends only on the number of fixed points \cite{Fulman2021,KimLee2020}, or only on the number of fixed points and two cycles \cite{Curien,hamaker,zbMATH07252777}. For general functions, some results have been proven  for example in \cite{FerayRandomPermutationsCumulants,zbMATH07502136}.

Nevertheless, relatively few results are known for large deviations for general conjugacy invariant permutations. In some sense, this question is related with large deviations for random matrices since \cite{kurt} shows that the law of the longest subsequence in a permutation chosen at random has the same limiting distribution as the largest eigenvalue of a Gaussian Wigner matrices. We will show that large deviations for random permutations can easily be seen to be  universal in a rather wide class of conjugacy invariant permutations  with a  sharp control over the total number of cycles. The main technique that we will introduce in this article is based on coupling of random permutations and  exponential approximations \cite[Section 4.2.2]{DZ}, see section \ref{tech}.

Hereafter, we will denote by $\kS_n$ the set of permutations of $\{1,\ldots,n\}=[n]$.
To state more precisely our result, let us remind the definition of conjugacy invariant permutations. 
\begin{definition} \label{Def: CI}
A random permutation $\sigma_n$ of size $n$ is said to be \emph{conjugacy invariant} if for every $\rho\in \kS_n$,  \[\rho\sigma_n\rho^{-1} \overset{d}{=} \sigma_n.\] 
\end{definition}
 In other words, 
$\sigma_n$ is conjugacy invariant if and only if the map $\sigma \in \kS_n\mapsto \P(\sigma_n=\sigma)$ depends only on the cycle structure of $\sigma$.
This class of permutations first emerged in biology. It  includes well-known measures like the Ewens measures (see  Definition \ref{defEwens}), along with various generalizations, such as the Kingman virtual permutations \cite{EWENS197287,Kingman1,TSILEVICH1998StationaryMO}.   In this article,  we are interested in the following class of conjugacy-invariant permutations with a  sharp control over the total number of cycles described as follows. 

\begin{definition} Let $0<\alpha,\beta\le 1$. 
    We say that a sequence  of random permutations $(\sigma_n)_{n\ge 1}$ satisfies $(\mathcal{CI}_{\alpha,\beta})$ if 
\begin{itemize}
    \item for any $n$, $\sigma_n$ is a conjugacy invariant permutation of size $n$,
\item and, for any $\varepsilon>0,$
\begin{align}\label{eq:cycle_controle}
\tag{$\mathcal{CI}_{\alpha,\beta}$}
      \lim_{n\to\infty} \frac{\ln{\P\left( \frac{\# \sigma_n}{n^\alpha} >\varepsilon \right)}}{n^\beta}  = -\infty,
    \end{align}
where $\#(\sigma)$ is the total number of cycles of $\sigma$. 
\end{itemize}
\end{definition}
The main goal of this article will be to show that the sets  $(\mathcal{CI}_{\alpha,\beta})$ provide natural universality classes for the large deviation of the uniform law on permutations for some appropriate choices of $\alpha$ and $\beta$. 

We observe that for any $\beta\le \alpha$, $\mathcal{CI}_{\alpha,\beta}\subset\mathcal{CI}_{\alpha,\alpha}$. 
A classical example of random permutations satisfying  $\CI{\alpha,\alpha}$ (for any $\alpha$) is the uniform permutation. More generally, we prove in section \ref{cialphaewens}, see Corollary~\ref{Cor:Ewens cycles},  that
the Ewens measures satisfy  $\CI{\alpha,\alpha}$ for any positive real number $\alpha$.   Let $\theta$ be a non-negative real number. We say that a random  permutation $\sigma^{\mathrm{Ew}}_{\theta,n}$  follows the Ewens distribution with parameter $\theta$ if for every $\sigma \in \mathfrak{S}_n$,
\begin{align}\label{defEwens}
\mathbb{P}\left(\sigma^{\mathrm{Ew}}_{\theta,n}=\sigma\right)= \frac{\theta^{\#(\sigma)-1}}{\prod_{i=1}^{n-1}(\theta+i)}.
\end{align}
In particular,  $\sigma^{\mathrm{Ew}}_{1,n}=\sigma_n^{\rm Unif}$ is the uniform permutation and 
$\sigma^{\mathrm{Ew}}_{0,n}$ is a uniform cyclic permutation. 
Clearly, Ewens measures  are conjugacy invariant since by definition
$\mathbb{P}(\sigma^{\mathrm{Ew}}_{\theta,n}=\sigma)$ depends only on the number of cycles of $\sigma$.

Our primary emphasis is the application of our universality results to monotone subsequences of random permutations.


\subsection{Monotone subsequences}
Let $\mathfrak{S}_n$ be the set of permutations of size $n$. Given $\sigma \in \mathfrak{S}_n$, a subsequence $(\sigma(i_1),\dots,\sigma(i_k))$ is an increasing (resp. decreasing) subsequence of $\sigma$ of length $k$ if $i_1<\dots<i_k$ and $\sigma(i_1)<\dots<\sigma(i_k)$ (resp. $\sigma(i_1)>\dots>\sigma(i_k)$). We denote by $\LIS(\sigma)$ (resp. $\mathrm{LDS}(\sigma)$) the length of the longest increasing (resp. decreasing) subsequence of $\sigma$. For example,
\begin{equation*}
\text{if}\quad \sigma=\begin{pmatrix}
1& 2 & 3 & 4 & 5 \\
5& 3 & 2 & 4 & 1
\end{pmatrix} ,  \,\,\, \LIS(\sigma)=2\,\,\text{ and } \,\, \mathrm{LDS}(\sigma)=4.
\end{equation*}

The study of the asymptotic behavior of monotone subsequences within random permutations is famously known as Ulam's problem.  In his seminal work \cite{MR0129165}, Ulam posed the conjecture that the limit, as $n$ tends towards infinity, of 
$
\frac{\mathbb{E}(\LIS(\sigma^{\text{Unif}}_n))}{\sqrt{n}}
$ exists.
Vershik, Kerov, Logan, and Shepp \cite{MR0480398,LOGAN1977206} proved that this limit is equal to $2$. { The limiting fluctuations are the same as that of some random matrices models \cite{Baik,BaikRains}. }
For historical details and full proofs, we refer to \cite{MR3468738}. 
Large deviations of speed $\sqrt{n}$ and $n$ were also  proved when $\sigma_n=\sigma^{\text{Unif}}_n$ follows the uniform law on $\kS_n$: 
\begin{theorem} \cite[Theorem~2]{MR1653841}\cite{deuschel_zeitouni_1999} \label{Thm: LD_uniform_LIS}

For any $x \ge 2$,
\[
\lim_{n\to\infty}
\frac{1}{\sqrt n} \ln{\P\left(\LIS\left(\sigma^{\text{Unif}}_n\right) \ge x\sqrt{n}\right)}= -I_{\LIS,\frac{1}{2}}(x) \]
and for any  $0<x<2$,
\[
\lim_{n\to\infty}
\frac{1}{n} \ln{\P\left(\LIS\left(\sigma^{\text{Unif}}_n\right) \le x\sqrt{n}\right)}= -I_{\LIS,{1}}(x)
\]
where 
\[I_{\LIS,\frac{1}{2}}(x)=2x\mathrm{cosh}^{-1}\left(\frac{x}{2}\right) 
\quad \text { and } \quad 
I_{\LIS,1}(x)= -1+\frac{x^2}4 + 2\ln\(\frac{x}{2}\) -\(2+\frac{x^2}{2}\) \ln\( \frac{2x^2}{4+x^2}\) . \]
\end{theorem}

Note that the large deviations to the right (for $x\ge 2$) and to the left (for $x<2$) have different speed. This is very similar to the large deviations principles for the largest eigenvalue of the Gaussian matrix ensembles \cite{MajS,BADG,BAG97}. The large deviations to the left  necessitate to move all the spectrum and in fact are related to the large deviations of the empirical measure of the eigenvalues : their speed is $n^2$, typically, the number of random entries of the matrices \cite{BAG97}. Whereas large deviations to the right are easier as they only require to move one eigenvalue, namely the largest, which can be achieved for instance by making one very large entry, which corresponds to a speed of order $n$. Therefore, these two results are, in fact, rather different in nature. 

Our first result, proven in Section \ref{PGDLIS}, is the following.
\begin{theorem} \label{Main_theorem_LIS}
    If the sequence of random permutations $(\sigma_n)_{n\ge 1}$  satisfies 
    $\CI{\frac 12,\frac 12}$,   then  the law of $\frac{\LIS(\sigma_n)}{\sqrt n}$ satisfies a large deviation principle with speed $\sqrt{n}$ and good rate function $J_{\LIS,\frac 12}$ which is equal to $I_{\LIS,\frac 12}$ on $[2,+\infty)$ and $+\infty$ on $(-\infty,2)$. In other words, 
    for any measurable subset $E$ of $ \mathbb{R}$, we have :

\begin{align} \label{eq : CI_LIS 1/2}
     -\inf_{x\in E^\circ} J_{\LIS,\frac 12}(x) \le 
\liminf_{n\to\infty}
\frac{\ln{\P\left(\frac{\LIS(\sigma_n)}{\sqrt n}\in E \right)}}{\sqrt n} \le 
\limsup_{n\to\infty}
\frac{ \ln{\P\left(\frac{\LIS(\sigma_n)}{\sqrt n}\in E \right)}}{\sqrt n}
\le -\inf_{x\in \bar{E}} J_{\LIS,\frac 12}(x).
\end{align}
Moreover,  if the sequence  of random permutations $(\sigma_n)_{n\ge 1}$ satisfies 
    $\CI{\frac 12,1}$,   then the law of $\frac{\LIS(\sigma_n)}{\sqrt n}$ satisfies a large deviation principle with speed $n$ and good rate function $J_{\LIS,I}$ 
    where 
$$
J_{\LIS,1}(x)=\begin{cases} I_{\LIS,1}(x) & 0< x \le2
\\ 0 & x> 2
\\ +\infty & x \leq 0

\end{cases}.
$$
In other words, 
     for every measurable subset  $E$ of $\mathbb{R}$,
\begin{align} \label{eq : CI_LIS 1}   
    -\inf_{x\in E^\circ} J_{\LIS,1}(x) \le 
\liminf_{n\to\infty}
\frac{\ln{\P\left(\frac{\LIS(\sigma_n)}{\sqrt n}\in E \right)}}{ n} \le 
\limsup_{n\to\infty}
\frac{\ln{\P\left(\frac{\LIS(\sigma_n)}{\sqrt n}\in E \right)}}{n} 
\le -\inf_{x\in \bar{E}} J_{\LIS,1}(x),
\end{align}

The same result holds if we replace $\LIS$ par $\LDS$, the length of the longest decreasing subsequence.
\end{theorem}


Even if Ewens measures do not satisfy  $\CI{\frac 12, 1}$ (except for $\theta=0)$ so that \eqref{eq : CI_LIS 1} is not guaranteed, we will show by a separate argument that universality still holds in the sense that the Ewens distribution still satisfies the same large deviation principle, see section \ref{cialphaewens}:

\begin{proposition} \label{Prop: Ewens_LIS}
For any $\theta\ge 0 $,  the sequence of random permutations  following the Ewens distribution  $\sigma_n\overset d= \sigma^{\mathrm{Ew}}_{\theta,n}$, $n\ge 1$, satisfies  the large deviation principle  \eqref{eq : CI_LIS 1}. 
\end{proposition}
\subsection{Strategy of the proofs}\label{tech}
In this section, we introduce our main technical trick, namely Theorem ~\ref{Prob:Meta theorem}, which can be thought as an exponential approximation argument as presented in \cite[Section 4.2.2]{DZ}. To this end, we see the statistic under study, such as $\LIS$,  as a function $f$ of the permutation. Our point is then that if this function $f$ depends sufficiently smoothly on the permutation in the sense of Hypothesis $\mathcal{H}_1$ and the law of $n^{-\alpha}f(\sigma_n^{\text{Unif}})$ satisfies a large deviation principle with speed $n^\beta$, then this large deviation principle remains true for all the random  permutations satisfying $\CI{\alpha,\beta}$. We then show that many well known statistics verify hypothesis  $\mathcal{H}_1$ and then apply our result to various examples. 

The goal of this section is to introduce a tool to prove universality  for the large and moderate deviations bounds   for some statistics.  We first define for a permutation $\sigma\in\mathfrak{S}_n$,  the set
\begin{equation*}
    A_\sigma =  \begin{cases} 
    \{\sigma\} & \text{if } \#(\sigma)=1\\ 
      \{\rho \in \mathfrak{S}_n, \rho=\sigma\circ(i_1,i_2)\circ(i_1,i_3)\dots\circ(i_1,i_{\#(\sigma)}) \text{ and } \#(\rho)=1 \} &\text{if } \#(\sigma)>1 
    \end{cases}.
\end{equation*}
Above, $\{i_1
,\ldots, i_{\#\sigma}\}$ are distinct elements of $\{1,\ldots,n\}$. 

For example: $A_{(3,4)} = \{(1,2,3,4), (1,2,4,3), (1,3,4,2), (1,4,3,2) \}$.
Indeed, $(1,2,3,4)=(3,4)(4,2)(4,1)$, 
$(1,2,4,3)=(3,4)(3,2)(3,1)$, etc.

Our main theorem states that large deviations are universal within the class $\CI{\alpha,\beta}$.

\begin{theorem}
 \label{Prob:Meta theorem}
Let $f$ be a function defined on $\bigcup_{n\ge 1} \mathfrak{S}_n$ and having values in $\mathbb{R}^d$ for some fixed integer number $d$. We denote $|.|$ the Euclidean distance in $\mathbb R^d$. 
Suppose that 
\begin{itemize}
    \item the following holds true
    \begin{equation} \label{Eq:function_control} \tag{$\mathcal{H}_1$} 
   \sup_{\sigma\in \cup_{n\geq 1}\mathfrak{S}_n} \sup_{\rho\in A_{\sigma}}   \frac{|f(\sigma)-f(\rho)|}{\#(\sigma)} <+\infty. 
    \end{equation}
     
\item there exists some function $J$ \textbf{continuous},  some $0<\beta\leq \alpha<1$ and  $\mathcal{I}$ an open set of  $\mathbb{R}^d$ such that  for any $x\in \mathcal{I}$
\begin{equation} \label{Eq:LDH} \tag{$\mathcal{H}_2$} 
    \lim_{n\to\infty}
\frac{1}{n^\beta} \ln{\P\(f\(\sigma^{\text{Unif}}_{n}\) \ge xn^\alpha\)}= -J(x).
\end{equation}
The relation $a=(a_1,\dots,a_d)\geq  b=(b_1,\dots,b_d)$ means $a_i\geq b_i$ for all $i\in [d]$.  
\item  Moreover, the sequence $(\sigma_n)_{n\ge}$ of random permutations satisfies 
    $\CI{\alpha,\beta}$
\end{itemize}
Then,  for any $x\in \mathcal{I}$
\begin{equation*}
    \lim_{n\to\infty}
\frac{1}{n^\beta} \ln{\P(f(\sigma_{n}) \ge xn^\alpha)}= -J(x).
\end{equation*}
\end{theorem}

It may appear challenging to verify \eqref{Eq:function_control} for an arbitrary statistic. However, in some cases, it is more straightforward to prove the following condition:
\begin{equation}
\label{Eq:function_control_2}
\tag{$\mathcal{H}_1'$}
\sup_{\sigma\in \cup_{n\geq 1}\mathfrak{S}n} \sup_{i,j } {|f(\sigma)-f(\sigma\circ(i,j))|} <+\infty.
\end{equation}

A noteworthy observation is that \eqref{Eq:function_control_2} implies \eqref{Eq:function_control}, as it directly follows from the triangle inequality.
\begin{lemma}\label{RK:vector space} The set of functions satisfying  $\eqref{Eq:function_control}$
(resp. \eqref{Eq:function_control_2} ) is a vector space.
\end{lemma}

We next show that Hypothesis $\eqref{Eq:function_control_2}$ is satisfied for a large class of well-known statistics. Hence, Theorem ~\ref{Prob:Meta theorem} implies the universality of the large deviations for these statistics as soon as they are known for the uniform measure. We prove the following Property in subsection~\ref{subsection:proof_remark}.
\begin{property} \label{RQ:controle_functions}
The following functions satisfy $\eqref{Eq:function_control_2}$ :
\begin{enumerate}
    \item The longest increasing subsequence  $\LIS$.
    \item  The longest decreasing subsequence $\LDS$.
    \item  The vector $(\lambda_1,\dots,\lambda_d)$ if $\lambda_i$  denotes the length of the $i^{\text{th}}$ row of the  RSK image for $i\in [d]$. $d$ is a fixed integer number. 
    \item The normalized Inversions count : $\frac{\mathrm{Inv}(\sigma)}{n}=\frac{\mathrm{card(\{ (i,j) : i<j, \sigma(i)>\sigma(j)  \})}}{n}$.
        \item The descents count : $D(\sigma):=\mathrm{card}\{ i: \sigma(i+1) <\sigma(i)  \}$.
                \item The ascents count : $A(\sigma):=\mathrm{card}\{ i: \sigma(i+1) >\sigma(i)  \}$.
        
            \item The peaks count :
$\mathrm{Peaks}(\sigma)=\mathrm{card}\{ i: \sigma(i-1) <\sigma(i) > \sigma(i+1)  \}$.
\item The valleys count :
$\mathrm{Valleys}(\sigma)=\mathrm{card}\{ i: \sigma(i-1) >\sigma(i) < \sigma(i+1)  \}$.
\item The exceedance count  : $\mathrm{Exc}(\sigma):=\mathrm{card}\{ i: \sigma(i) >i  \}$.  \footnote{ For the specialists, in general, the normalized  number of occurrences of any set of classical (inversions for example), consecutive (descents, ascents, double descents, peaks, valleys, etc.  )  or more generally (bi-)vincular pattern.}
        \item The normalized major index:  $\frac{\mathrm{Maj}(\sigma) }{n}=\frac{\sum_{\sigma(i+1)>\sigma(i)} i  }{n}$.
        \item The longest alternating subsequence $\LAS$: Let $1\le i_1<i_2<\cdots<i_k\le$. 
        We say that $\sigma(i_1),\dots,\sigma(i_k)$ is an alternating subsequence of length $k$ if
        $\sigma(i_1)<\sigma(i_2)>\sigma(i_3)<\sigma(i_4) \cdots$ and let $\LAS$ be the length of the longest alternating subsequence.
\end{enumerate}
\end{property}
In (iii), $\lambda(\sigma)=(\lambda_i(\sigma))_{i\geq 1}$ denotes the RSK shape associated with the permutation $\sigma$, see a definition in \cite[Section 1.6]{MR3468738}

\subsection{Applications}
\subsubsection{Applications to the longest increasing subsequence.}

A direct application of Theorem ~\ref{Prob:Meta theorem} gives the upper large deviations \eqref{eq : CI_LIS 1/2} of Theorem~\ref{Main_theorem_LIS}. 

\begin{corollary}\label{Thm:LDS_right_eq}
    If the sequence $(\sigma_n)_{n\ge 1}$ of random permutations  satisfies 
    $\CI{\frac 12,\frac 12}$,   then for any $x>2$ 
    \begin{eqnarray*}
\lim_{n\to\infty}
\frac{1}{\sqrt n} \ln{\P(\LIS(\sigma_n) \ge x\sqrt{n})}&= &-I_{\LIS,\frac 12}(x)\\
\lim_{n\to\infty}
\frac{1}{\sqrt n} \ln{\P(\LDS(\sigma_n) \ge x\sqrt{n})}&=&
-I_{\LIS,\frac 12}(x).
\end{eqnarray*}
\end{corollary} 
\begin{proof}
The first point  is a direct application  of Theorem~\ref{Prob:Meta theorem} by taking $f=\LIS$  and $\alpha=\beta=\frac{1}{2}$.  The hypothesis \eqref{Eq:function_control} is satisfied thanks to  Property~\ref{RQ:controle_functions},  
and \eqref{Eq:LDH} is satisfied thanks to  Theorem~\ref{Thm: LD_uniform_LIS}. Moreover, when $\sigma_n$ follows the uniform law, the LDS has the same distribution as the LIS. This can be seen for instance by replacing the collection $\{i\}$ into $\{n-i\}$.The proof of the large deviation principle for the function $\LDS$ is therefore the same as for the function $\LIS$. 

\end{proof} 
For the upper moderate deviation, we find the following universal result:
\begin{corollary}    
Let $ \frac{1}{6}<\nu <\frac{1}{2}$. If  the sequence $(\sigma_n)_{n\ge 1}$ of random permutations satisfies 
    $\CI{ \nu,\frac{3\nu}{2}-\frac{1}{4}}$,   then for any $ x>0$,
    \[
\lim_{n\to\infty}
\frac{1}{n^{\frac{3\nu}{2}-\frac{1}{4}}} \ln{\P(\LIS(\sigma_n) \ge 2\sqrt{n}+  xn^{\nu}} )= -\frac{4}{3} x^{\frac{3}{2}}.\]
\end{corollary}
This result was proven in the uniform case in \cite{MD_LIS_upper}. Our universality result therefore follows from  Theorem~\ref{Prob:Meta theorem} since the coefficients $(\alpha,\beta)= (\nu, \frac{3\nu}{2}-\frac{1}{4})$ satisfy $\beta=\frac{3\nu}{2}-\frac{1}{4} \le \nu=\alpha$ for every  $\nu\le 1/2$ so that  \eqref{Eq:LDH} is satisfied. 
\eqref{Eq:function_control} follows from Property~\ref{RQ:controle_functions}.

\subsubsection{Applications to large deviations for the Eulerian statistics}
Let $D$ be the number of descents in $\sigma$. 
\begin{definition}
A function $f$ is called an
Eulerian statistic if $f(\sigma^{\mathrm{Unif}}_n) \overset{d}{=} D(\sigma^{\mathrm{Unif}}_n) $ for every integer number $n$.
    
\end{definition}
It is known, for example that  the exceedances count $\mathrm{Exc}$ and the ascents count  $A$ are Eulerian. Note that the equality in distribution is not true for general conjugacy invariant permutations.  


\begin{corollary}    
    If the sequence of random permutations $(\sigma_n)_{n\ge 1}$ satisfies 
    $\CI{ 1,1}$,   then for $f \in \{D,A,\mathrm{Exc}\}$, for any $ \frac{1}2<x<1$, 
    \[
\lim_{n\to\infty}
\frac{1}{n} \ln{\P(f(\sigma_n) \ge x{n})}= 
-I_{D}(x).
\]
 Where $I_{D}(x)=\sup_t\{ xt-\ln(\frac{\exp(t)-1}t)\}.$
\end{corollary} 
\begin{proof}
For the descents, this LDP is already known in the uniform case \cite{bercu2022sharp} which implies the same large deviation  principle for the other Eulerian statistics readily as they have the same distribution.  Hence,   \eqref{Eq:LDH} is satisfied with $\alpha=\beta=1$. 
Universality follows again from Theorem ~\ref{Prob:Meta theorem}, since Property ~\ref{RQ:controle_functions} implies that $ D,A,\mathrm{Exc}$ satisfy  \eqref{Eq:function_control}.
\end{proof}

\subsubsection{Applications to the joint  large deviations for the  descents and  the inverse descents}

    Let $ D $ be the number of descents in a permutation $ \sigma $, and let $ D' $ be the number of {inverse descents}, defined by  
\[
D'(\sigma) = D(\sigma^{-1}).
\]

\begin{corollary}    
    If the sequence of random permutations $ (\sigma_n)_{n \ge 1} $ satisfies  
    $ \CI{1,1} $, then for any $ \frac{1}{2} < x, x' < 1 $,  
    \[
    \lim_{n \to \infty}
    \frac{1}{n} \ln{\P\left(D(\sigma_n) \ge x n,\; D'(\sigma_n) \ge x' n\right)} = 
    -I_D(x) - I_D(x').
    \]
\end{corollary}

\begin{proof}
    This large deviation principle is already known in the uniform case, thanks to a recent work of Bercu et al.~\cite{fredes2024sharpanalysisjointdistribution}.  
    Hence, equation~\eqref{Eq:LDH} is satisfied with $ \alpha = \beta = 1 $.  
    Universality then follows again from Theorem~\ref{Prob:Meta theorem}. In fact, Property~\ref{RQ:controle_functions} implies that $ D $ satisfies \eqref{Eq:function_control}.  
For $ D' $, observe that for any $ \sigma \in S_n $ and any transposition $ (i,j) $,
    \[
    \begin{aligned}
    \left| D'(\sigma \circ (i,j)) - D'(\sigma) \right| 
    &= \left| D((i,j) \circ \sigma^{-1}) - D(\sigma^{-1}) \right| \\
    &= \left| D(\sigma^{-1} \circ (\sigma(i), \sigma(j))) - D(\sigma^{-1}) \right|.
    \end{aligned}
    \]
    Using Property~\ref{RQ:controle_functions}, for descents, $ D' $ satisfies also \eqref{Eq:function_control_2}, and thus the vector $ (D, D') $ satisfies \eqref{Eq:function_control}.
\end{proof}

\subsubsection{Edge of RSK}

In this subsection, we give an additional result, which  is not a direct application of Theorem ~\ref{Prob:Meta theorem} but which proof uses the same techniques. 

We are interested in the lower tail of $\LIS$, and more precisely the length of the first rows of the RSK image.

\begin{proposition} \label{Pro:speend n}
    If the sequence $(\sigma_n)_{n\ge}$ of random permutations satisfies 
    $\CI{\frac 12, 1}$,   then for any $0<x_d<\dots<x_1<2$,  
\[
\limsup_{n\to\infty}
\frac{1}{ n} \ln{\P( \forall i \in [d], \,\, \lambda_i(\sigma_n) \le  x_i\sqrt{n})}=
-I_{\LIS,1}(x_d).
\]
\end{proposition}   
To the best of our knowledge, this result is not stated in the uniform case  for general integer number $d$ but it is immediate to adapt the proof of \cite{deuschel_zeitouni_1999} to get the uniform case, see Section \ref{secProp13}.

\subsection{Some comments}
Except for the special case of the uniform permutation, there is a lack of existing results regarding the large deviations of general conjugacy-invariant permutations. 
Even for the uniform case, the large deviation theory of many statistics has not been studied yet. For example, we are not aware of existing results for the large deviations for $\LAS$, the  permutations patterns counts, and the upper edge of RSK in the uniform case. Even if the hypotheses \eqref{Eq:function_control_2} is satisfied for a large family of functions, since we are using comparison techniques, the application of Theorem~\ref{Prob:Meta theorem} is not possible without knowing the large deviations  in the case of  uniform permutations.

{ In the context of random matrices, two distinct classes of large deviation results emerge: those concerning the edge of the spectrum and those concerning the the global spectral measure. The methodologies employed in proving these results typically differ significantly.

Our results primarily align with edge results. The analogue to the spectral measure in the context of random permutations is the shape of the RSK image. Using our techniques, one could investigate the universality of the shape. However,  choosing the appropriate topology may be challenging. }

In the literature, some non-universality results have been established for certain statistics but for other families of random permutations. For instance, \cite{deuschel_zeitouni_1999} studied  the longest increasing subsequence of i.i.d. points sampled from a measure on the unit square. Additionally, \cite{pinsky2023large} explores the large deviations of $\LIS$ and $\LAS$ for permutations uniformly chosen among those that avoid a pattern of length $3$. Moreover, in \cite[Theorem B]{meliot2022asymptotics}, a large deviation principle  for the major index has been proven for a distinct family of random permutations.

{

A permutation can be viewed from two perspectives: as a matrix or as a word. In the context of Ewens permutations, their interpretation as random matrices reveals that many spectral properties are non-universal, i.e. depends on $\theta$ as the size of the permutation, goes to infinity \cite{BeArDa,BaNa}. However, when considered as random words, numerous statistics become universal in the limit  \cite{FerayRandomPermutationsCumulants,zbMATH07502136}. This work can be seen as an extension of the latter results to large deviations.

While our proof uses a key argument of \cite{zbMATH07502136}, the context of large deviations necessitates new ideas, particularly concentration inequalities, to ensure exponential approximations. Moreover, for the specific case of  the large deviations of the longest increasing subsequence at speed $n$, we introduce novel probabilistic arguments, notably   Lemma~\ref{Lem:decomposition_LIS_ping} and \eqref{Eq:equivalent 18}.

}
\section{Proof of the results}
\label{sec:proof}

\subsection{Proof of Property~\ref{RQ:controle_functions}}
\label{subsection:proof_remark}
In this section, we prove that many natural statistics satisfy $\eqref{Eq:function_control_2}$.
\begin{itemize}
\item {\it Monotone subsequences ($\LIS$, $\LDS$, and $\lambda_i$):}
The cases of $\LIS$ and $\LDS$ have already been demonstrated in \cite[Lemma 3.1]{kammoun2018}, and the proof for $\lambda_i$ is presented in \cite[Lemma 3.4]{kammoun2018}.
\item  {\it Inversions:}

Let $\sigma \in \mathfrak{S}_n$ be a fixed permutation, and let $1 \leq i_1 < j_1 \leq n$. Define $\rho = \sigma \circ (i_1, j_1)$.

The key observation is that for all $i \notin \{i_1, j_1\}$, we have $\rho(i) = \sigma(i)$. The remaining of the proof varies depending on the specific statistic considered, but the underlying idea remains the same. We will provide detailed explanations for a  few of these statistics.

    Inversions can be expressed as follows:
    $$\mathrm{Inv}(\sigma)= |\{(i, j) : i < j, \sigma(i) > \sigma(j)\}|\,.$$
    We make the following decomposition: 
    \begin{align*}
   \underbrace{\{(i, j) : i < j, \sigma(i) > \sigma(j)\}}_{s}
    &= \underbrace{\{(i, j) : i < j, \sigma(i) > \sigma(j), \{i, j\} \cap \{i_1, j_1\} \neq \emptyset\}}_{s_1} \\
    &\quad \cup \underbrace{\{(i, j) : i < j, \sigma(i) > \sigma(j), \{i, j\} \cap \{i_1, j_1\} = \emptyset\}}_{s_2}.
    \end{align*}
    
    Furthermore, for inversions of $\rho$, we have:
    
    \begin{align*}
    \underbrace{\{(i, j) : i < j, \rho(i) > \rho(j)\}}_{s'}
    &= \underbrace{\{(i, j) : i < j, \rho(i) > \rho(j), \{i, j\} \cap \{i_1, j_1\} \neq \emptyset\}}_{s'_1} \\
    &\quad \cup \underbrace{\{(i, j) : i < j, \rho(i) > \rho(j), \{i, j\} \cap \{i_1, j_1\} = \emptyset\}}_{s'_2}.
    \end{align*}

    Using the key observation, one can see that $s'_2 = s_2$, which allows us to write:
    
    \begin{align*}
    |\mathrm{Inv}(\sigma) - \mathrm{Inv}(\rho)| = |\mathrm{card}(s) - \mathrm{card}(s')| 
    &= |\mathrm{card}(s'_1) - \mathrm{card}(s_1)| \\
    &\leq  \max(\mathrm{card}(s'_1), \mathrm{card}(s_1)) \\
    &\leq 2n.
    \end{align*}
As a consequence,
$$\frac{1}{n}|\mathrm{Inv}(\sigma) - \mathrm{Inv}(\rho)|\le 2$$
which proves that $n^{-1} \mathrm{Inv}$ satisfies $\eqref{Eq:function_control_2}$.
    \item {\it Descents  and Major index:}
     The same idea applies here.  
     
     Recall that $D(\sigma)=|\{i : \sigma(i+1) < \sigma(i)\}|$. 
     Let
    \begin{align*}
    \underbrace{\{i : \sigma(i+1) < \sigma(i)\}}_{s}
    &= \underbrace{\{i : \sigma(i+1) < \sigma(i), \{i, i+1\} \cap \{i_1, j_1\} \neq \emptyset\}}_{s_1} \cup \underbrace{\{i : \sigma(i+1) < \sigma(i), \{i, i+1\} \cap \{i_1, j_1\} = \emptyset\}}_{s_2}
    \end{align*}
    and
    \begin{align*}
    \underbrace{\{i : \rho(i+1) < \rho(i)\}}_{s'}
    &= \underbrace{\{i : \rho(i+1) < \rho(i), \{i, i+1\} \cap \{i_1, j_1\} \neq \emptyset\}}_{s'_1}  \cup \underbrace{\{i : \rho(i+1) < \rho(i), \{i, i+1\} \cap \{i_1, j_1\} = \emptyset\}}_{s'_2}.
    \end{align*}
    Then, we have
    \begin{align*}
       |D(\sigma) - D(\rho)| \leq \max(\mathrm{card}(s_1) , \mathrm{card}(s'_1)) \leq 4.
   \end{align*}
   Similarly, with $\mathrm{Maj}(\sigma)=\sum_{\sigma(i+1)>\sigma(i)} i$, we find 
   \begin{align*}
       |\mathrm{Maj}(\sigma) - \mathrm{Maj}(\rho)| = \left|\sum_{i\in s_1}i -\sum_{i\in s'_1}i\right| \le \max\left(\sum_{i\in s_1} i,\sum_{i\in s'_1} i \right)\leq n \max(\card(s_1),\card(s'_1))    \le 4n.
   \end{align*}
   \begin{itemize}
       \item
For {\it Peaks and Valleys}, the proof is similar to that of descents.

\item {For longest alternating subsequence $\LAS$}, we can use the following characterization 
(see    \cite{houdre2010} and   \cite[Corollary 2]{MR2820763}) 
$$\mathrm{LAS}(\sigma) = 1+ \sum_{i=1}^{n-1} M_k(\sigma),$$
where
$M_1(\sigma)=\mathbbm{1}_{\sigma(1)>\sigma(2)}$
and for $1<k<n$,
$$ M_k(\sigma)= \mathbbm{1}_{\sigma(k-1)>\sigma(k)<\sigma(k+1)} +\mathbbm{1}_{ \sigma(k-1)<\sigma(k)>\sigma(k+1) }. $$
   \end{itemize}
\end{itemize}We have that
$\mathrm{LAS} = M_1 + \mathrm{Valleys} + \mathrm{Peaks}$.
$M_1$ satisfies \eqref{Eq:function_control_2} since it is a bounded function. Therefore, by using Lemma~\ref{RK:vector space}, $\mathrm{LAS}$ also satisfies \eqref{Eq:function_control_2}. 

\subsection{Proof of Property $\CI{\alpha,\alpha}$ for Ewens distributions }\label{cialphaewens}
For the Ewens distribution, it is known that the law of the  total number of cycles has a nice description as a sum of independent Bernoulli variables. 
\begin{proposition} \label{Prop: Ewens cycles law}
The number of cycles of  $\sigma^{\mathrm{Ew}}_{\theta,n}$ is equal in distribution to $\sum_{i=1}^n a_{\theta,i}$  where $(a_{\theta,i})_{i}$ are independent Bernoulli variables with $\P(a_{\theta,i}=1)=\frac{\theta}{i+\theta-1}$.
\end{proposition}
This property can be proved using the Chinese restaurant process description of Ewens permutations.  It is a classical result,  we can cite for example \cite[equation (1.1)]{Ch2013} and \cite{MR3458588}. 
Many concentration inequalities are known for the sum of independent variables.  For our purpose, we use a  special form of the Bennett’s inequality. 
\begin{proposition}\cite[Theorem 2.9]{Concentration}\label{Prob: Benett}
\footnote{The version we use is\cite[Theorem 2.9]{Concentration} by setting $b=1$.}

Let $X_1,X_2,\dots,X_n$ be independent random variables such that almost surely $X_i \le 1$. Then 

\begin{align}
  \ln\left(  \P\left(\sum_{i=1}^n X_i- \E(X_i) > t   \right)\right)\le -(v+t)\ln(1+\frac tv)+t.
\end{align}
where $v=\sum_{i=1}^n{\E{X^2_i}}$.
\end{proposition}

\begin{corollary} \label{Cor:Ewens cycles}
The sequence    $(\sigma^{\mathrm{Ew}}_{\theta,n})_{n\geq 1}$ satisfies $\CI{\alpha,\alpha}$,
for every $\theta \ge 0 $, and every $\alpha>0$, \end{corollary}
\begin{proof}
    Using Proposition~\ref{Prop: Ewens cycles law}, we have 
$\#(\sigma^{\mathrm{Ew}}_{\theta,n})\overset d=\sum_{i=1}^n a_{\theta,i}$, and 
$$\sum_{i=1} ^n \E (a_{\theta,i}) = \sum_{i=1} ^n \E (a^2_{\theta,i})= \theta \ln (n)+O(1).$$
 We want to prove that for every $\varepsilon>0$, every $\alpha>0$, 
 \begin{align}
\lim_{n\to\infty}      \frac{\ln{\P\left( \frac{\# \sigma^{\mathrm{Ew}}_{\theta,n}}{n^\alpha} >\varepsilon \right)}}{n^\alpha}  = -\infty,
    \end{align}
  This is a direct consequence of Proposition \ref{Prop: Ewens cycles law} by setting  $X_i=a_{\theta,i}$, $t= \varepsilon n^\alpha-v$ in Proposition ~\ref{Prob: Benett}.
\end{proof}

\subsection{Proof of Theorem ~\ref{Prob:Meta theorem}}

In order to prove Theorem ~\ref{Prob:Meta theorem}, we need to introduce a one step Markov chain $T$. It is  the same  as in \cite{kam2}.  It maps a conjugacy invariant random permutations $\sigma_n$ to a permutation having the same law as $\sigma^{\mathrm{Ew}}_{0,n}$. 

This Markov chain does not change a lot statistics satisfying \eqref{Eq:function_control}.
Let  $T$  be the Markov chain defined on $\mathfrak{S}_n$ and associated to the stochastic matrix 
$\left[\frac{\mathrm{1}_{A_{\sigma}}(\rho)}{\mathrm{card}(A_{\sigma})}\right]_{\sigma,\rho \in \mathfrak{S}_n}$ where we recall that
\begin{equation*}
    A_\sigma =  \begin{cases} 
    \{\sigma\} & \text{if } \#(\sigma)=1\\ 
      \{\rho \in \mathfrak{S}_n, \sigma^{-1}\circ\rho=(i_1,i_2)\circ(i_1,i_3)\dots\circ(i_1,i_{\#(\sigma)}) \text{ and } \#(\rho)=1 \} &\text{if } \#(\sigma)>1 
    \end{cases}.
\end{equation*}

$T$ is then the Markov operator mapping a permutation $\sigma$ to a permutation uniformly chosen at random  among the permutations obtained by merging the cycles of $\sigma$ using transpositions having all a common point. 

\begin{lemma}{\cite[Lemma~6]{kam2}} \label{Lemma: coupling}
For any  conjugacy invariant random permutation $\sigma_n$ on $\mathfrak{S}_n$,

$$T(\sigma_n) \overset d= \sigma^{\mathrm{Ew}}_{0,n}.$$
\end{lemma}
The proof is detailed in \cite[Lemma~6]{kam2} but the idea is rather simple.
By construction, $T(\sigma_n)$ has almost surely one cycle. 
Since permutations with one cycle  belong to the same conjugacy class, it is sufficient to prove that $T(\sigma_n)$ is conjugacy invariant as soon as $\sigma_n$ is conjugacy invariant.  


We move now to the proof of Theorem~\ref{Prob:Meta theorem}
\begin{proof}[Proof of Theorem{ \ref{Prob:Meta theorem}}]
We consider a function $f$ on $\cup_{n\geq 1}\mathfrak{S}_n$ with values in $\mathbb R^d$.
One can suppose for simplicity that
 \begin{equation}\label{p0}  \sup_{\sigma\in \cup_{n\geq 1}\mathfrak{S}_n} \sup_{\rho\in A_{\sigma}}   \frac{|f(\sigma)-f(\rho)|}{\#(\sigma)} \le 1. 
 \end{equation}
 Otherwise one can apply the theorem to $\frac{f}{\sup_{\sigma\in \cup_{n\geq 1}\mathfrak{S}_n} \sup_{\rho\in A_{\sigma}}   \frac{|f(\sigma)-f(\rho)|}{\#(\sigma)}}$ since by hypothesis the denominator is finite.
 Let $\sigma_{n}$ be conjugacy invariant.
By Lemma~\ref{Lemma: coupling}, 
we know that 
\begin{equation}\label{p1}\P(f (\sigma^{\mathrm{Ew}}_{0,n}) \ge xn^\alpha)= \P(f (T(\sigma_{n}))  \ge xn^\alpha ).\end{equation}
Moreover, by \eqref{p0}, we have
$$   \P(f(\sigma_{n})  \ge xn^\alpha + \#\sigma_n \mathbf{1} ) \le \P(f (T(\sigma_{n}) ) \ge xn^\alpha ) \le   \P(f (\sigma_{n})  \ge xn^\alpha - \#\sigma_n\mathbf{1} ) \,. $$
Here, $\mathbf{1}$ is the vector of $\mathbb{R}^d$ with all components equal to $1$. 
Let $\varepsilon>0$. We write the following decomposition  
\begin{align*}
  p^\pm:=  \P(f (\sigma_{n})  \ge xn^\alpha \pm \#\sigma_n \mathbf{1} ) =&  \underbrace{\P(f (\sigma_{n})  \ge xn^\alpha \pm \#\sigma_n \mathbf{1}  | \#(\sigma_{n}) < \varepsilon n^\alpha)}_{p^{\pm}_1} \underbrace{ \P(\#(\sigma_{n}) < \varepsilon n^\alpha)}_{p_2} 
    \\&+ \underbrace{\P(f (\sigma_{n})  \ge xn^\alpha \pm \#\sigma_n\mathbf{1}  | \#(\sigma_{n}) \ge
    \varepsilon n^\alpha )}_{p^{\pm}_3} \underbrace{\P(\#(\sigma_{n}) \ge  \varepsilon n^\alpha)}_{p_4}.
\end{align*}
Moreover, 
$$  p^-_1 \le   \P(f(\sigma_{n})  \ge (x-\varepsilon \mathbf{1}) n^\alpha   | \#(\sigma_{n}) < \varepsilon n^\alpha ) 
 $$
gives readily that 
$$p^-_1 p_2 \le  \P(f (\sigma_{n})  \ge (x-\varepsilon\mathbf{1})  n^\alpha   \text{ and  }  \#(\sigma_{n}) < \varepsilon  n^\alpha )  \le  \P(f (\sigma_{n})  \ge (x-\varepsilon \mathbf{1})  n^\alpha).$$
Consequently, we find that 
\begin{equation}\label{p2} p^{-} \le  \P(f (\sigma_{n})  \ge (x-\varepsilon\mathbf{1}) n^\alpha ) + p_4.\end{equation}
Similarly,
$$  p^+_1 \ge   \P(f (\sigma_{n})  \ge (x+\varepsilon\mathbf{1})  n^\alpha   | \#(\sigma_{n}) < \varepsilon  n^\alpha  ) 
 $$
and  then 
\begin{align}
p^+ \ge p^+_1 p_2 &\ge  \P(f (\sigma_{n})  \ge (x+\varepsilon\mathbf{1})  n^\alpha    \text{ and  }  \#(\sigma_{n}) < \varepsilon  n^\alpha  )  \nonumber
\\&= \P(f (\sigma_{n})  \ge (x+\varepsilon \mathbf{1})  n^\alpha) -  \P(f (\sigma_{n})  \ge (x+\varepsilon \mathbf{1})  n^\alpha   \text{and  }  \#(\sigma_{n}) \ge \varepsilon  n^\alpha  ) \nonumber
\\&\ge   P(f (\sigma_{n})  \ge (x+\varepsilon \mathbf{1})  n^\alpha  -p_4.\label{p3}
\end{align}
To sum-up, for any conjugacy invariant permutation $\sigma_n$, for any $\varepsilon>0$, \eqref{p0},\eqref{p1},\eqref{p2},\eqref{p3} imply 
\begin{align} \label{eq:controle LIS-Ew}
\P(f (\sigma_{n})  \ge (x+\varepsilon \mathbf{1})  n^\alpha ) - \P(\#(\sigma_{n}) \ge  \varepsilon  n^\alpha )
&\le  \P(f(\sigma^{\mathrm{Ew}}_{0,n}) \ge x n^\alpha ) 
\\&\le \P(f (\sigma_{n})  \ge (x-\varepsilon\mathbf{1})  n^\alpha ) + \P(\#(\sigma_{n}) \ge  \varepsilon  n^\alpha ) \nonumber.
\end{align}
We next choose $\sigma_n=\sigma^{\text{Unif}}_{n}$ to be the uniform permutation. Because the  Ewens distribution with $\theta=1$ is the uniform distribution, Corollary~\ref{Cor:Ewens cycles} implies that $\sigma^{\text{Unif}}_{n}$ is $\CI{\alpha,\alpha}$. Therefore, for any $\varepsilon>0$, for $n$ large enough we have under the hypothesis of Proposition ~\ref{Prob:Meta theorem} and because $\beta\le\alpha$, we find for all $M>0$ 
$$
    \P(f (\sigma^{\text{Unif}}_{n} )  \ge (x+\varepsilon \mathbf{1}) n^\alpha) -\P(\#(\sigma^{\text{Unif}}_{n}) \ge \varepsilon  n^\alpha )= \exp(-n^\beta (J(x+\varepsilon \mathbf{1})) +o(1))- o( \exp(-n^\beta M))$$
    If $J(x+\varepsilon \mathbf{1})$ is infinite, then  the right hand side will also be smaller than $\exp(-n^\beta M)$ for $M$ as large as wished, whereas if it is finite, taking $M>J(x+\varepsilon \mathbf{1})$ also gives
    $$
     \P(f (\sigma^{\text{Unif}}_{n} )  \ge (x+\varepsilon \mathbf{1}) n^\alpha) -\P(\#(\sigma^{\text{Unif}}_{n}) \ge \varepsilon  n^\alpha )
    =\exp(-n^\beta (J(x+\varepsilon \mathbf{1}) +o(1)))\,.
$$
Similarly 
\begin{align*}
    \P(f(\sigma^{\text{Unif}}_{n})  \ge (x-\varepsilon \mathbf{1}) n^\alpha) +\P(\#(\sigma^{\text{Unif}}_{n}) \ge  \varepsilon n^\alpha)= \exp(-n^\beta J(x-\varepsilon\mathbf{1})+ o(n^\beta ))(1+o(1)).
\end{align*}
We therefore conclude that, for every $\varepsilon>0$, 
\[-J(x+\varepsilon\mathbf{1}) \le
\liminf_{n\to\infty}
\frac{1}{n^\beta} \ln{\P(f(\sigma^{\text{Ew}}_{0,n}) \ge xn^\alpha)}\le 
\limsup_{n\to\infty}
\frac{1}{n^\beta} \ln{\P(f(\sigma^{\text{Ew}}_{0,n}) \ge xn^\alpha)} \le -J(x-\varepsilon\mathbf{1}).
\]
Consequently, since we assumed that $J$ is continuous, we find by letting $\varepsilon$ going to zero
\begin{equation}\label{liminf}\liminf_{n\to\infty}
\frac{1}{n^\beta} \ln{\P(f(\sigma^{\text{Ew}}_{0,n}) \ge xn^\alpha)}=
\limsup_{n\to\infty}
\frac{1}{n^\beta} \ln{\P(f(\sigma^{\text{Ew}}_{0,n}) \ge xn^\alpha)}= -J(x).
\end{equation}
Now let $\sigma_n$ be a conjugacy invariant permutation. Equation \eqref{eq:controle LIS-Ew} implies (by choosing first to replace $x$ by $x+\varepsilon \mathbf{1}$ then by $x-\varepsilon\mathbf{1}$) that

\begin{align}
\P(f (\sigma^{\mathrm{Ew}}_{0,n})  \ge (x+\varepsilon \mathbf{1}) n^\alpha) - \P(\#(\sigma_{n}) \ge  \varepsilon n^\alpha)
&\le  \P(f (\sigma_{n}) \ge xn^\alpha) \label{gen}
\\&\le \P(f(\sigma^{\mathrm{Ew}}_{0,n})  \ge (x-\varepsilon \mathbf{1}) n^\alpha) + \P(\#(\sigma_{n}) \ge  \varepsilon n^\alpha)\,. \nonumber
\end{align}
Under  hypothesis \eqref{eq:cycle_controle}, \eqref{liminf} implies that
\begin{align} \label{Eq:LD_f_1}
    \P(f (\sigma^{\mathrm{Ew}}_{0,n})  \ge (x-\varepsilon \mathbf{1}) n^\alpha) +\P(\#(\sigma_{n}) \ge  \varepsilon n^\alpha)= \exp(-n^\beta J(x-\varepsilon\mathbf{1})+ o(n^\beta ))(1+o(1)),
\end{align}
and 
\begin{align}
    \P(f (\sigma^{\mathrm{Ew}}_{0,n})  \ge (x+\varepsilon\mathbf{1}) n^\alpha) -\P(\#(\sigma_{n}) \ge  \varepsilon n^\alpha)= \exp(- n^\beta J(x+\varepsilon\mathbf{1})+ o(n^\beta))(1+o(1)).
\end{align}
Plugging these estimates in \eqref{gen}, letting $n$ going to infinity and then $\varepsilon$ going to zero (while using the continuity of $J$) completes the proof of Theorem ~\ref{Prob:Meta theorem} for permutations satisfying \eqref{eq:cycle_controle}.
\end{proof}

\subsection{{Proof of Theorem ~\ref{Main_theorem_LIS}}}\label{PGDLIS}
We first remark that it is enough to prove the Theorem with $E=[x, \infty)$ for $x\ge 2$ or $E=(0,x)$ if $x\le 2$. Indeed, for $x>2$, it is not hard to see that
$I_{\LIS,\frac{1}{2}}$ is strictly increasing  so that the probability that $\LIS$ belongs to $[x,+\infty)$ is equivalent to the probability that it belongs to $[x,x+\delta]$ for any $\delta>0$. Hence, \eqref{eq : CI_LIS 1/2} for $E=[x, \infty)$ and every $x\ge 2$ yields the weak large deviation principle. Exponential tightness is as well clear as $I_{\LIS,\frac{1}{2}}$ goes to $+\infty$ at infinity. Hence, Corollary \ref{Thm:LDS_right_eq} implies the full large deviation principle above $2$. Similarly,  it is easy to see that $I_{\LIS,1}$ is strictly decreasing on $(0,2)$ so that proving \eqref{eq : CI_LIS 1} for $E=(0,x]$ yields the full large deviation principle below $2$. However, proving \eqref{eq : CI_LIS 1} for $E=(0,x]$ for $x<2$ is more complicated 
because it is not possible to use Theorem ~\ref{Prob:Meta theorem} {since $1=\beta>\alpha=1/2$}. We need a proof specific to the longest increasing subsequence that we detail below. 
Before proving this result we prove a key lemma. 
   \begin{lemma} \label{Lem:decomposition_LIS_ping}
For every given permutation  $\sigma$, for every integer number $k$,

    $$ \mathbb{P}(\LIS(T(\sigma)) \leq  \LIS(\sigma)+ k) \geq   \inf_{\tau\in S_{\#(\sigma)}} \mathbb{P} (\LIS(\tau\circ\sigma^{\mathrm{Ew}}_{0,\#(\sigma)})<k) $$  
\end{lemma}

\begin{proof}[Proof]
The idea is to observe first that $\rho\in A_\sigma$ if and only if 
$\sigma^{-1}\circ\rho=(i_1,i_2,\ldots,i_{\#(\sigma)})$
where $i_j$ and $i_k$ are in different cycles of $\sigma$ as soon as $j\neq k$. This implies that one way to construct $T(\sigma)$ is to choose first uniformly,
$j_1<j_2<\ldots<j_{\#(\sigma)}$ each from one cycle of $\sigma$, $\pi$ a uniform  permutation of size $\#(\sigma)$. Then, it is easy to see that  $T(\sigma)$ has  the same law as 

$$\sigma \circ (j_1,j_{\pi(1)},
j_{\pi^2(1)}
\ldots, j_{\pi^{\#(\sigma)-1}(1)}).$$

Fix now $j_1,j_2,\ldots,j_{\#(\sigma)}$ each on a cycle of $\sigma$ and $\pi$ a cyclic permutation
and let  $\ell_1<\dots <\ell_{\LIS(T(\sigma))}$ be such that $T(\sigma)(\ell_1)<\ldots< T(\sigma)(\ell_{\LIS(T(\sigma))})$ be a maximal increasing subsequence of $T(\sigma)$.
Let $$E= \{ j_1,j_2,\ldots,j_{\#(\sigma)} \} \cap 
 \{ \ell_1,\ell_2,\ldots,\ell_{\LIS(T(\sigma))}\}
=\{j_{a_1},j_{a_2}, j_{a_{\mathrm{card(E)}}}\}$$ 
with $a_1<a_2...<a_{\mathrm{card}(E)}$ and let $F= \{ \ell_1,\ell_2,\ldots,\ell_{\LIS(T(\sigma))} \}
   \setminus E.$
For any $\ell_k \in F$,  $T(\sigma)(\ell_k)=\sigma(\ell_k)$ and then $\mathrm{card}(F)\leq \LIS(\sigma)$.
Let $\tau$  be the unique permutation of $\{1,\dots\#(\sigma)\}$ 
such that 
$$\sigma (j_{\tau^{-1}(1)})<  \sigma (j_{\tau^{-1}(2)}) <\ldots < \sigma (j_{\tau^{-1}(\#(\sigma))}).$$
Moreover,
 $T(\sigma)(j_{a_k})=  \sigma(j_{\pi({a_k})})$ 
 and then 
$$ \sigma(j_{\pi(a_1)}) <  \ldots <T(\sigma)(j_{\pi(a_{\mathrm{card}(E)})}) $$
In particular, the following holds true: 
$$\mathrm{card}(E) \leq \LIS(\tau\pi).$$ 
Consequently, we find that
$$\LIS(T(\sigma))= \mathrm{card(E)}+\mathrm{card(F)} \leq \LIS(\sigma)+ \LIS(\tau\circ\pi).$$
Now, applying this inequality we find that for any integer number $k$,

\begin{eqnarray*}\mathbb{P}( \LIS(T(\sigma)) -   \LIS (\sigma) \leq k | \tau) &\ge & \mathbb{P} (\LIS(\tau\circ \pi)<k)\\
&= &\mathbb{P} (\LIS(\tau\circ\sigma^{\mathrm{Ew}}_{0,\#(\sigma)})<k) \ge  \inf_{\tau'\in S_{\#(\sigma)}} \mathbb{P} (\LIS(\tau'\circ\sigma^{\mathrm{Ew}}_{0,\#(\sigma)})<k) \end{eqnarray*}
where in the second line we used that  $\pi$ follows the  uniform cyclic permutation of length $\#(\sigma)$.
\end{proof}



To prove the second part of Theorem \ref{Main_theorem_LIS}, we first show the result for the Ewens distribution with $\theta=0$, namely that for every $x<2$, 

\begin{align} \label{Eq:Ewens_lower}\liminf_{n\to\infty}
\frac{1}{ n} \ln{\P(\LIS(\sigma^{\mathrm{Ew}}_{0,n}) \le  x\sqrt{n})}=
    \limsup_{n\to\infty}
\frac{1}{ n} \ln{\P(\LIS(\sigma^{\mathrm{Ew}}_{0,n}) \le  x\sqrt{n})}=
-I_{\LIS,1}(x).
\end{align}
It is then straightforward to generalize this result to $\sigma_n\in \CI{\frac 12, 1}$ 
as in the proof of Theorem ~\ref{Prob:Meta theorem} by taking $f=-\LIS$, $\alpha=\frac{1}{2}$ and $\beta=1$. 

The upper bound is trivial since for any $k\in \mathbb{N}$,

$$\P(\LIS(\sigma^{\mathrm{Ew}}_{0,n}) =k)= \frac{\rm{card}\{\sigma  : \LIS(\sigma)=k, \#(\sigma)=1 \}}{(n-1)!} \le \frac{\rm{card}\{\sigma  : \LIS(\sigma)=k \}}{(n-1)!}
= n \P(\LIS(\sigma^{\text{Unif}}_{n}) =k).
$$ 
Consequently, by Theorem \ref{Thm: LD_uniform_LIS}, we deduce
\[\limsup_{n\to\infty} \frac{1}{ n} \ln{\P(\LIS(\sigma^{\mathrm{Ew}}_{0,n}) \le x\sqrt{n})} \le 
\limsup_{n\to\infty} \frac{1}{ n} (\ln{\P(\LIS(\sigma^{\text{Unif}}_{n}) \le x\sqrt{n})}+\ln(n))=
-I_{\LIS,1}(x).
\]
{  
The lower bound is more sophisticated. 
Fix $x<2$, $0<\varepsilon<x$, and $\sigma\in \mathfrak{S}_n$. We assume that $\sigma$ is such that  $\LIS(\sigma)< (x-3\varepsilon)\sqrt{n}$ and $\#(\sigma) < \varepsilon^2 n$. Then, we find that  
\begin{align}
    \mathbb{P}(\LIS(T(\sigma)) \leq x\sqrt{n})
    &= \mathbb{P}(\LIS(T(\sigma)) \leq (x-3\varepsilon)\sqrt{n} + 3\varepsilon\sqrt{n}) \nonumber\\
    &\geq \mathbb{P}(\LIS(T(\sigma)) \leq \LIS(\sigma)+ 3\varepsilon\sqrt{n}) \label{p5}
    \end{align}

By Lemma~\ref{Lem:decomposition_LIS_ping}, \eqref{p5} gives
    
    \begin{align} \label{eq:bound T sigma fixed}
    \mathbb{P}(\LIS(T(\sigma)) \leq x\sqrt{n})
    &\geq \inf_{\tau\in S_{\#(\sigma)}} \mathbb{P}(\LIS(\tau\circ\sigma^{\mathrm{Ew}}_{0,\#(\sigma)})<3\varepsilon\sqrt{n}).
\end{align}

To conclude, we need the following easy lemma.
\begin{lemma} \label{Lem:product_ewens_LIS}
For any $x>2$,
$$\lim_{n\to\infty}\inf_{\tau\in \mathfrak{S}_n} \mathbb{P} (\LIS(\tau\circ\sigma^{\mathrm{Ew}}_{0,n})<x\sqrt{n}) =1 $$
\end{lemma}
\begin{proof}
The distribution of 
$\tau\circ\sigma^{\mathrm{Ew}}_{0,n}$ is  uniform over a subset of size exactly $(n-1)!$ of $\mathfrak{S}_n$ (the image $S_\tau$ by $\tau$ of the cyclic permutations). This set  has probability $1/n$ for $\sigma^{\text{Unif}}_{n}$. We have then
$$\mathbb{P} (\LIS(\tau\circ\sigma^{\mathrm{Ew}}_{0,n})\geq x\sqrt{n}) =n\mathbb{P} (\{\LIS(\sigma_n^{\text{Unif}})\ge x\sqrt{n}\}\cap \{\sigma^{\text{Unif}}_{n}\in \mathfrak{S}_\tau\})\le n \mathbb{P} (\{\LIS(\sigma_n^{\text{Unif}})\ge x\sqrt{n})=o(1)$$
because of Theorem \ref{Thm: LD_uniform_LIS} and $x>2$. Note that these bounds do not depend on $\tau$. 
Therefore
$$ \mathbb{P} (\LIS(\tau\circ\sigma^{\mathrm{Ew}}_{0,n})<x\sqrt{n}) =  1-\mathbb{P} (\LIS(\tau\circ\sigma^{\mathrm{Ew}}_{0,n})\geq x\sqrt{n})   = 1-o(1) \,.$$

\end{proof}



Consequently, for any $x<2$ and $\varepsilon>0$, and $\sigma_n$ conjugacy invariant, we have
\begin{align}
\nonumber \P(\LIS(\sigma^{\mathrm{Ew}}_{0,n}) \le x\sqrt{n})
&= \sum_{\sigma\in \mathfrak{S}_n} \P(\LIS(T(\sigma)) \le x\sqrt{n})\mathbb{P}(\sigma_n=\sigma)  
\\ \nonumber &\ge \sum_{\LIS(\sigma) \le (x-3\varepsilon)\sqrt{n}, \#(\sigma) <\varepsilon^2n}  \nonumber
\inf_{\tau\in \mathfrak{S}_{\#(\sigma)}} 
\mathbb{P} (\LIS(\tau\circ\sigma^{\mathrm{Ew}}_{0,k})<3\varepsilon\sqrt{n}) \mathbb{P}(\sigma_n=\sigma)  
  \\ &\ge \inf_{k\le \varepsilon^2n} \inf_{\tau\in \mathfrak{S}_{k}} \mathbb{P} (\LIS(\tau\circ\sigma^{\mathrm{Ew}}_{0,k})<3\varepsilon\sqrt{n})   \P(\LIS(\sigma_{n}) \le (x-3\varepsilon)\sqrt{n}, \#(\sigma_{n}) <\varepsilon^2n )  \nonumber
  \\
  &\ge (1+o(1)) \P(\LIS(\sigma_{n}) \le (x-3\varepsilon)\sqrt{n}, \#(\sigma_{n}) <\varepsilon^2n ) \label{Ineq:lower_conjugacy_invariant}
 \end{align}
 In the first line we used Lemma \ref{Lemma: coupling}, in the second we restricted the summation over $\sigma$ and bounded uniformly from below the first probability by using \eqref{eq:bound T sigma fixed}, in the third line we summed the second probability over the remaining permutations, and in the last line we used Lemma \ref{Lem:product_ewens_LIS} when $k$ goes to infinity (noticing that $\varepsilon \sqrt{n}/\sqrt{k}\ge 1$ for $n$ large enough), while the bound is clear when $k$ is finite. 
Now choose $\sigma_n$ to be uniform. We deduce from \eqref{p5} that
\begin{align*}
\P(\LIS(\sigma^{\mathrm{Ew}}_{0,n}) \le x\sqrt{n}) 
  &\ge (1+o(1)) \P(\LIS(\sigma^{\text{Unif}}_{n}) \le (x-3\varepsilon)\sqrt{n}, \#(\sigma^{\text{Unif}}_{n}) <\varepsilon^2n ) \\
  &= (1+o(1)) (\P(\LIS(\sigma^{\text{Unif}}_{n}) \le (x-3\varepsilon)\sqrt{n}) \\
  &\quad - \P(\LIS(\sigma^{\text{Unif}}_{n}) \le (x-3\varepsilon)\sqrt{n}, \#(\sigma^{\text{Unif}}_{n}) \geq \varepsilon^2n )) \\
  &\ge (1+o(1)) \P(\LIS(\sigma^{\text{Unif}}_{n}) \le (x-3\varepsilon)\sqrt{n}) = e^{-n I_{\LIS,\frac{1}{2}}(x-\varepsilon)(1+o(1))}
\end{align*}
where we used Theorem \ref{Thm: LD_uniform_LIS}
and that $\sigma^{\text{Unif}}_{n}$ is $\CI{\alpha,\alpha}$. We finally can let $n$ going to infinity and $\varepsilon$ going to zero to get 
 \eqref{Eq:Ewens_lower}. To complete the proof of  Theorem ~\ref{Main_theorem_LIS}, one only needs to check  that
\begin{itemize}
        \item     If the sequence $(\sigma_n)_{n\ge 1}$ of random permutations satisfies 
    $\CI{\frac 12,1}$, for any $\varepsilon>0$, for any $x\geq2$.
    $$ \lim_{n\to\infty} \frac{1}{n} \ln\(\mathbb{P}\(\LIS(\sigma_n) \in \((x-\varepsilon)\sqrt{n},(x+\varepsilon)\sqrt{n}\)\)\)= 0.$$
Since $\CI{\frac 12, 1}$ implies  $\CI{\frac 12,\frac 12}$,    this is a direct consequence of Corollary~\ref{Thm:LDS_right_eq}.
    
    \item     If  the sequence $(\sigma_n)_{n\ge 1}$ satisfies 
    $\CI{\frac 12,\frac 12}$, for any $0<y<2$.
    $$ \lim_{n\to\infty} \frac{\ln(\mathbb{P}(\LIS(\sigma_n) < y\sqrt n))  }{\sqrt{n}} = -\infty$$
    
Indeed, by taking $f=-\LIS,\, \varepsilon=(2-y)/2, \,  x=-(y+ \varepsilon)$ and $\alpha=\beta=\frac{1}{2}$ in \eqref{eq:controle LIS-Ew}, the first inequality becomes

\begin{align*}
\P(\LIS (\sigma_{n})  \leq y \sqrt{n} ) \le  \P(\#(\sigma_{n}) \ge  \varepsilon  \sqrt{n} )
+  \P(\LIS(\sigma^{\mathrm{Ew}}_{0,n}) \leq x \sqrt{n} )). 
\end{align*}
The first term goes to zero faster than $e^{-M\sqrt{n}}$ for any $M>0$,  since the sequence $(\sigma_n)_{n\ge 1}$ satisfies 
    $\CI{\frac 12,\frac 12}$, whereas the second term goes to zero 
\end{itemize}


\subsection{Proof of Proposition~\ref{Prop: Ewens_LIS}}

For the upper tail, one can apply directly Corollary~\ref{Thm:LDS_right_eq}.

For the lower bound of the lower tail, let $\theta >0$.
\begin{align*}
\P(\LIS(\sigma^{\mathrm{Ew}}_{\theta,n}) =k)&\geq 
\P(\LIS(\sigma^{\mathrm{Ew}}_{\theta,n}) =k,\#\sigma^{\mathrm{Ew}}_{\theta,n}=1)\\&=
\frac{\rm{card}\{\sigma  : \LIS(\sigma)=k, \#(\sigma)=1 \}\Gamma(\theta)}{\Gamma(n+\theta)} 
\\&= \frac{\Gamma(\theta)\Gamma(n)}{\Gamma(n+\theta)} \P(\LIS(\sigma^{\text{EW}}_{0,n}) =k)
\geq
\frac{\Gamma(\theta)}{(n+\theta)^\theta} \P(\LIS(\sigma^{\text{EW}}_{0,n}) =k)
.
\end{align*}
This implies by taking $k\le x \sqrt{n}$, 
\[\liminf_{n\to\infty} \frac{1}{ n} \ln{\P(\LIS(\sigma^{\mathrm{Ew}}_{\theta,n}) \le x\sqrt{n})} \ge 
\liminf_{n\to\infty} \frac{1}{ n} \ln{\P(\LIS(\sigma^{\mathrm{Ew}}_{0,n}) \le x\sqrt{n})}=
-I_{\LIS,\frac{1}{2}}(x).
\]

For the upper bound,  one can conclude directly  by \eqref{Ineq:lower_conjugacy_invariant}, by choosing $\sigma_n$ to be $\sigma^{\mathrm{Ew}}_{\theta,n}$.

\subsection{Proof of Proposition~\ref{Pro:speend n}}\label{secProp13}
We will adapt the proof of the lower tail of Theorem~\ref{Main_theorem_LIS}. We choose to give two different separate proofs for readability reasons. 

 First remark that  Proposition~\ref{Pro:speend n} is equivalent to that 
 \[
\limsup_{n\to\infty}
\frac{1}{ n} \ln{\P\( \forall j \in [d], \,\,  \sum_{i=1}^j\lambda_i(\sigma_n) \le  \sum_{i=1}^j {x_i}\sqrt{n}\)}=
-I_{\LIS,1}(x_d).
\]
Indeed, because the right hand side  depends only on $x_d$, we see that mostly the deviations of $\lambda_d(\sigma_n)$ matters. The same phenomenon appears for random matrices: the probability fix the $d$ largest eigenvalues to make a deviation below $2$ is equivalent to the probability that the spectrum stays below the smallest one, namely $x_d$. 

We recall that according to Green's theorem (\cite[Theorem 3.5.3]{Sagan2001})  $\lambda_1+\dots,\lambda_k$ is the maximum sum of lengths of $k$ disjoint increasing subsequences.

The counterpart of Lemma~\ref{Lem:decomposition_LIS_ping} is  that 
for every  $\sigma$ deterministic, for every integer numbers  $d,k$,

\begin{align} 
\label{Eq:equivalent 18}
\mathbb{P}\(\sum_{i=1}^d\lambda_i(T(\sigma)) \leq  \(\sum_{i=1}^d \lambda_i(\sigma)\)+ k\) \geq   \inf_{\tau\in \mathfrak{S}_{\#(\sigma)}} \mathbb{P} \(\sum_{i=1}^d \lambda_i(\tau\circ\sigma^{\mathrm{Ew}}_{0,\#(\sigma)})<k\).\end{align}

As in the proof of Theorem~\ref{Main_theorem_LIS}, we  want to prove that

\begin{align} \label{Eq:Ewens_lower2}\liminf_{n\to\infty}
\frac{1}{ n} \ln{\P\(\forall j \in [d], \sum_{i=1}^j \lambda_i(\sigma^{\mathrm{Ew}}_{0,n}) \le  \sum_{i=1}^j x_i\sqrt{n}\)}&=
    \limsup_{n\to\infty}
\frac{1}{ n} \ln{\P\(\forall j \in [d], \sum_{i=1}^j \lambda_i(\sigma^{\mathrm{Ew}}_{0,n}) \le  \sum_{i=1}^j x_i\sqrt{n}\)}\\&=
-I_{\LIS,1}(x_d). \nonumber
\end{align}
It is then straightforward to generalize this result to $\sigma_n\in \CI{\frac 12, 1}$ 
as in the proof of Theorem ~\ref{Prob:Meta theorem} by taking $f=(\lambda_1,\lambda_1+\lambda_2,\cdots,\sum_{j=1}^d\lambda_j)$, $\alpha=\frac{1}{2}$ and $\beta=1$. 

The upper bound is trivial for the same reason that in the proof of Theorem~\ref{Main_theorem_LIS} (the probability of any event under Ewens with parameter $0$ is at most $n$ times its probability under the uniform permutation).

The lower bound is more sophisticated. 
Fix $0<x_d<\dots<x_1<2$, $0<\varepsilon<x_d$, and $\sigma\in S_n$. We assume that $\sigma$ is such that for any $j$,  $\sum_{i=1}^j \lambda_j(\sigma)< ((\sum_{i=1}^j x_i)-3\varepsilon)\sqrt{n}$ and $\#(\sigma) < \frac{\varepsilon^2 n}{d^2}$. Then, we find that  
\begin{align}
    \mathbb{P}\(\forall j \in [d], \sum_{i=1}^j \lambda_i(T(\sigma)) \leq \sum_{i=1}^j x_i\sqrt{n}\)
    &= \mathbb{P}\left(\forall i \in [d],  \sum_{i=1}^j \lambda_i(T(\sigma)) \leq \left(\left(\sum_{i=1}^jx_i\right)-3\varepsilon\right)\sqrt{n} + 3\varepsilon\sqrt{n}\right) \nonumber\\
    &\geq \mathbb{P}\left(\forall  j \in [d], \sum_{i=1}^j \lambda_i(T(\sigma)) \leq \left(\sum_{i=1}^j \lambda_i(\sigma)\right)+ 3\varepsilon\sqrt{n}\right) \label{p'5}
    \end{align}

By \eqref{Eq:equivalent 18}, \eqref{p'5} gives
    
    \begin{align} \label{eq:bound T sigma fixed_2}
    \mathbb{P}\(\forall j \in [d],   \sum_{i=1}^j \lambda_i(T(\sigma)) \leq \sum_{i=1}^j x_i\sqrt{n})\)
    &\geq \inf_{\tau\in \mathfrak{S}_{\#(\sigma)}} \mathbb{P}\(\forall j \in [d], \sum_{i=1}^j \lambda_i\(\tau\circ\sigma^{\mathrm{Ew}}_{0,\#(\sigma)}\)<3\varepsilon\sqrt{n}\).
\end{align}

To conclude, we need the following easy lemma.
\begin{lemma} \label{Lem:product_ewens_LIS_2}
For any $x>2$ and any integer number  $j$,
$$\lim_{n\to\infty}\inf_{\tau\in \mathfrak{S}_n} \mathbb{P} \(\sum_{i=1}^j \lambda_i(\tau\circ\sigma^{\mathrm{Ew}}_{0,n})<j x\sqrt{n}\) =1 $$
\end{lemma}
\begin{proof}
First we have,
\begin{align*}
    \lim_{n\to\infty}\inf_{\tau\in \mathfrak{S}_n} \mathbb{P} \(\sum_{i=1}^j \lambda_i(\tau\circ\sigma^{\mathrm{Ew}}_{0,n})<j x\sqrt{n}\) \ge
    \lim_{n\to\infty}\inf_{\tau\in \mathfrak{S}_n} \mathbb{P} \(j \LIS(\tau\circ\sigma^{\mathrm{Ew}}_{0,n})<j x\sqrt{n}\)\end{align*}
because of the non-increasing of $\lambda_i$.
Therefore, one can conclude by Lemma \ref{Lem:product_ewens_LIS}  that 
$$\lim_{n\to\infty}\inf_{\tau\in \mathfrak{S}_n}\mathbb{P} \(j \LIS(\tau\circ\sigma^{\mathrm{Ew}}_{0,n})<j x\sqrt{n}\)= \mathbb{P} \( \LIS(\tau\circ\sigma^{\mathrm{Ew}}_{0,n})< x\sqrt{n}\)=1+o(1).$$
\end{proof}



Consequently, for any $\sigma_n$ conjugacy invariant, similarly to \eqref{Ineq:lower_conjugacy_invariant}, we obtain 
\begin{multline}
 \P\( \forall j \in [d], \sum_{i=1}^j \lambda_i(\sigma^{\mathrm{Ew}}_{0,n}) \le \sum_{i=1}^j x_i\sqrt{n}\) \\\ge (1+o(1)) \P\(\forall j \in [d], \sum_{i=1}^j \lambda_i(\sigma_{n}) \le  \(\(\sum_{i=1}^j x_i\)-3\varepsilon\)\sqrt{n}, \#(\sigma_{n}) <\frac{\varepsilon^2n}{d^2} \) \label{Ineq:lower_conjugacy_invariant2}.
 \end{multline}
The remaining of the proof is identical to that  of the lower tail of Theorem~\ref{Main_theorem_LIS}.



\section{Further discussions}
We start by discussing the optimality of our conditions. 
{The condition $\CI{ \frac{1}{2},\frac{1}{2}}$ for $\LIS$ is optimal as a condition on cycles.  Indeed, let us construct a sequence of random permutations $\sigma_n$ such that $\#\sigma_n/\sqrt{n}$ is of order  $x\neq 0$ with probability of order $e^{\sqrt{n}C}$ for some $C$ finite and show that our result does not apply to $\sigma_n$.  
In fact, let 
$\sigma_n$ be a permutation
constructed as follows.
With probability $e^{-(I_{\LIS,1/2}(x)-\varepsilon)\sqrt{n}}$, 
 $\sigma_n$ has $\lfloor x\sqrt{n} \rfloor$ fixed points, the other points belonging to a cycle  of length  $n-\lfloor x\sqrt{n} \rfloor$. With probability 
$1-e^{-(_{\LIS,1/2}(x)-\varepsilon)\sqrt{n}}$, $\sigma_n$ is a  uniform cyclic permutation.
In this case, it is not difficult to check that since the fixed points furnish an increasing subsequence 

    \[
\lim_{n\to\infty}
\frac{1}{\sqrt n} \ln{\P(\LIS(\sigma_n) \ge x\sqrt{n})}\ge 
-I_{\LIS,1/2}(x) +\varepsilon>-I_{\LIS,1/2}(x)\,,
\]
}
showing that $\LIS(\sigma_n)$ can not follow the same large deviation principle than $\LIS(\sigma_n^{\text{Unif}})$ as stated in Theorem \ref{Thm: LD_uniform_LIS}.

When the number of cycles is not sufficiently controlled, it is only possible to obtain one bound. 
\begin{proposition}\label{pro:LIS one bound}

    If $(\sigma_n)_{n\ge}$ satisfies 
    $\CI{ 1,1}$,   then for any $0<x<2$,

\begin{equation}\label{tol}
\limsup_{n\to\infty}
\frac{1}{ n} \ln{\P(\LIS(\sigma_n) \le x\sqrt{n})}\le
-I_{\LIS,1}(x).
\end{equation}

    Moreover, if  $(\sigma_n)_{n\ge}$ satisfies that for any $\varepsilon >0$
    
    $$\liminf_{n\to\infty} \mathbb P\(\frac{\#(\sigma_n)}{n} < \varepsilon\) >0, $$
       then for any $x>2.$ 
    \[
\liminf_{n\to\infty}
\frac{1}{\sqrt n} \ln{\P(\LIS(\sigma_n) \ge x\sqrt{n})} \ge 
-I_{\LIS,\frac 12}(x).
\]
\end{proposition} 

\begin{proof}[Sketch of the proof of Proposition~\ref{pro:LIS one bound}]
We will compare directly  $\LIS(\sigma_n)$ to $\LIS(\sigma^{\mathrm{Ew}}_{0,n})$. 
For the first inequality, \eqref{Ineq:lower_conjugacy_invariant}  and \eqref{eq : CI_LIS 1} imply that 
\begin{align}
   e^{-n I_{\LIS,1}(x) (1+o(1)) }&\ge 
\P(\LIS(\sigma^{\mathrm{Ew}}_{0,n}) \le x\sqrt{n}) \label{tolo}
   \\&\ge (1+o(1)) \P(\LIS(\sigma_{n}) \le (x-3\varepsilon)\sqrt{n}, \#(\sigma_{n}) <\varepsilon^2n ).\nonumber
  \end{align}

Moreover, if $\sigma_n $ satisfies $\CI{ 1,1}$, for every $M>0$ and $n$ large enough, 
$$\P(\LIS(\sigma_{n}) \le (x-3\varepsilon)\sqrt{n}, \#(\sigma_{n}) <\varepsilon^2n )\ge    \P(\LIS(\sigma_{n}) \le (x-3\varepsilon)\sqrt{n}) - e^{-nM}.$$
which gives \eqref{tol} with \eqref{tolo}.

The second inequality is an application of \cite[Lemma~21]{zbMATH07502136}.
First remark that a conjugacy invariant permutation conditioned on 
$\#(\sigma_n)\le k$ is still conjugacy invariant. One needs only to prove this result for permutations  where the number of cycles is less than 
$\varepsilon n $ for $n \geq n_0$ almost surely.

By choosing $\rho$ to be $\sigma^{\mathrm{Ew}}_{0,n}$ and  $k=1$ in \cite[Lemma~21]{zbMATH07502136},  one can reformulate the lemma to obtain that for any conjugacy invariant permutation $\sigma_n$, there exists $\widehat{\sigma_n}\overset{d}{=}\sigma_n$ such that, for any $1\le a<b$ and for any $c>0$   
$$\mathbb{E} \(\(\LIS(\widehat{\sigma_n})-\LIS(\sigma^{\mathrm{Ew}}_{0,n}) \)_{-} \middle| \#\widehat{\sigma}_n <c,\LIS(\sigma_{0,n}^{Ew}) \in[a,b]\) \leq \frac{cb}{n}. $$
By Markov inequality, and by taking $a=x\sqrt{n}+4x\varepsilon\sqrt{n}$ and $b=2x\sqrt{n}$ and $c=\varepsilon n$ we obtain that

\begin{align*}
    \mathbb{P} \(\LIS({\sigma_n}) \geq x\sqrt{n} \)
    &=     \mathbb{P} \(\LIS(\widehat{\sigma_n}) \geq x\sqrt{n} \) \\&\geq 
    \mathbb{P}\(\LIS(\sigma^{\mathrm{Ew}}_{0,n} \in [a,b],\(\LIS(\widehat{\sigma_n})-\LIS(\sigma^{\mathrm{Ew}}_{0,n}) \)_{-} \le 4x\varepsilon\sqrt{n}\)
    \\&=
    \mathbb{P}\(\(\LIS(\widehat{\sigma_n})-\LIS(\sigma^{\mathrm{Ew}}_{0,n}) \)_{-} \le 4x\varepsilon\sqrt{n} | \LIS(\sigma^{\mathrm{Ew}}_{0,n} \in [a,b],\) \mathbb{P}(\LIS(\sigma^{\mathrm{Ew}}_{0,n} \in [a,b])
    \\&\geq \(1-\frac{\frac{cb}{n}}{4x\varepsilon\sqrt{n}}\) I_{\LIS,\frac{1}{2}} \mathbb{P}(\LIS(\sigma^{\mathrm{Ew}}_{0,n} \in [a,b])
    =\frac{1}{2}\mathbb{P}(\LIS(\sigma^{\mathrm{Ew}}_{0,n} \in [a,b])
    \end{align*}
which concludes the proof.
\end{proof}


{\textbf{Acknowledgement}}

The second author would  acknowledge  useful discussions with Reda Chhaibi,  Valentin Féray  and Christian  Houdr\'{e}. He would also like to express gratitude to the organizers of the conference "High Dimensional Statistics and Random Matrices" where the initial idea for the project originated.
Both authors are  partially supported by ERC Project LDRAM (ERC-2019-ADG Project: 884584). The second author is partially supported by LabEx CIMI (ANR-11-LABX-0040) and a CNRS PEPS grant CAPRICES and ANR LOUCCOUM.

\bibliographystyle{abbrv}

\end{document}